\documentclass[reqno]{amsart}
\usepackage{amsthm,amsmath,amsfonts,amssymb,enumerate,latexsym,mathrsfs,lineno}

\usepackage{graphicx}
\newtheorem{thm}{Theorem}

\newtheorem*{thm*}{Theorem}

\newtheorem{defn}{Definition}
\newtheorem*{defn*}{Definition}
\newtheorem*{prop}{Proposition}

\newtheorem{claimm}{Claim}

\newtheorem{lemma}{Lemma}

\newtheorem*{cor*}{Corollary}

\newcommand{\drawat}[3]{\makebox[0pt][l]{\raisebox{#2}{\hspace*{#1}#3}}}

\newcommand{\N}{\mathbb{N}}
\newcommand{\R}{\mathbb{R}}
\newcommand{\Q}{\mathbb{Q}}

\newcommand{\iso}{\mathrm{iso}}
\newcommand{\funct}[2]{#1 \longrightarrow #2}

\newcommand{\restrict}[2]{#1\mspace{-2mu}\mathbin{\upharpoonright}\mspace{-1mu} #2}

\newcommand{\Ur}{\textbf{U}}

\newcommand{\s}{\textbf{S}}

\newcommand{\dom}{\mbox{$\mathrm{dom}$}}
\newcommand{\ran}{\mbox{$\mathrm{ran}$}}

\newcommand{\m}[1]{\textbf{#1}}

\newcommand{\mc}[1]{\widetilde{\textbf{#1}}}

\author{L. Nguyen Van Th\'{e} and N. W. Sauer}

\address{}

\email{nguyen@math.ucalgary.ca}

\email{nsauer@math.ucalgary.ca}

\title{The Urysohn sphere is oscillation stable.}

\date{April, 2007}

\begin{document}

\begin{abstract}

We solve the oscillation stability problem for the Urysohn sphere, an analog of the distortion problem for $\ell _2$ in the context of the Urysohn space $\Ur$. This is achieved by solving a purely combinatorial problem involving a family of countable ultrahomogeneous metric spaces with finitely many distances. 

\end{abstract}

\maketitle

\section{Introduction.}

The purpose of this article is to provide a combinatorial solution for an analog of the distortion problem for $\ell _2$. This latter problem can be formulated as follows: let $\mathbb{S} ^{\infty}$ denote the unit sphere of the Hilbert space $\ell _2$. Is it true that if $\varepsilon > 0$ and $f : \funct{\mathbb{S}^{\infty}}{\R}$ is uniformly continuous, then there is a closed infinite-dimensional subspace $V$ of $\ell
_2$ such that \[ \sup \{ \left| f(x) - f(y) \right| : x, y \in V \cap \mathbb{S} ^{\infty} \} < \varepsilon? \]

Equivalently, for a metric space $\m{X} = (X, d^{\m{X}})$, a subset $Y \subset X$ and $\varepsilon > 0$, let 
\[(Y) _{\varepsilon} = \{ x \in X : \exists y \in Y \ \ d ^{\m{X}} (x,y) \leqslant \varepsilon \}.\]

Then the distortion problem for $\ell _2$ asks: given a finite partition $\gamma$ of $\mathbb{S} ^{\infty}$, is there always $\Gamma \in \gamma$ such that $(\Gamma)_{\varepsilon}$ includes $V \cap \mathbb{S} ^{\infty}$ for some closed infinite-dimensional subspace $V$ of $\ell_2$? That problem appeared in the early seventies when Milman's work led to the following property, which is at the heart of Dvoretzky's theorem: 

\begin{thm*}[Milman \cite{Mil}]
\label{thm:Milman'} Let $\gamma$ be a finite partition of\/ $\mathbb{S} ^{\infty}$. Then for every
$\varepsilon > 0$ and every  $N \in \N$, there is $\Gamma \in \gamma$ and an $N$-dimensional subspace $V$ of $\ell _2$ such that $V \cap \mathbb{S}^{\infty} \subset (\Gamma)_{\varepsilon}$.
\end{thm*}

In that context, the distortion problem for $\ell _2$ really asks whether this result has an infinite dimensional analog. It is only a long time after Milman's theorem was established that the distortion problem for $\ell _2$ was solved by Odell and Schlumprecht in \cite{OS}:

\begin{thm*}[Odell-Schlumprecht \cite{OS}]

\label{thm:Odell-Schlumprecht}

There is a finite partition $\gamma$ of\/ $\mathbb{S} ^{\infty}$ and  $\varepsilon > 0$ such that no
$(\Gamma)_{\varepsilon}$ for $\Gamma \in \gamma$ includes $V \cap \mathbb{S} ^{\infty}$ for any closed infinite-dimensional subspace $V$ of $\ell_2$.

\end{thm*}

This result is traditionally stated in terms of the Banach space theoretic concept of oscillation stability, but can also be stated thanks to a new concept of oscillation stability for topological groups introduced by Kechris, Pestov and Todorcevic in \cite{KPT} (cf \cite{Pe1} for a detailed exposition). In this latter formalism, the theorem of Odell and Schlumprecht is equivalent to the fact that the standard action of $\iso (\mathbb{S}^{\infty})$ on $\mathbb{S}^{\infty}$ is not oscillation stable. On the other hand, in the context of isometry groups of complete separable ultrahomogeneous metric spaces, oscillation stability for topological groups coincides with the Ramsey-theoretic concept of \emph{approximate indivisibility}. Recall that a metric space is called \emph{ultrahomogeneous} when every isometry between finite
metric subspaces of $\m{X}$ can be extended to an isometry of $\m{X}$ onto itself. For
$\varepsilon \geqslant 0$, call a metric space $\m{X}$ \emph{$\varepsilon$-indivisible} when for
every finite partition $\gamma$ of $\m{X}$, there is $\Gamma \in \gamma$ and
$\mc{X} \subset \m{X}$ isometric to $\m{X}$ such that
\[\mc{X} \subset (\Gamma)_{\varepsilon}.\]

Then $\m{X}$ is \emph{approximately indivisible} when $\m{X}$ is $\varepsilon$-indivisible for
every $\varepsilon > 0$, and $\m{X}$ is \emph{indivisible} when $\m{X}$ is $0$-indivisible. 

Using this terminology, the theorem of Odell and Schlumprecht states that the sphere $\mathbb{S}^{\infty}$ is not approximately indivisible. However, because the proof is not based on the intrinsic geometry of $\ell _2$, the impression somehow persists that something is still missing in our understanding of the metric structure of $\mathbb{S}^{\infty}$. That fact was one of the motivations for \cite{LANVT} as well as for the present paper: our hope is that understanding the indivisibility problem for another remarkable space, namely the Urysohn sphere $\s$, will help to reach a better grasp of $\mathbb{S}^{\infty}$. The space $\s$ is defined as follows: up to isometry, it is the unique complete separable ultrahomogeneous metric space with diameter $1$ into
which every separable metric space with diameter less or equal to $1$ embeds isometrically. Equivalently, it is also the sphere of radius $1/2$ in the so-called universal Urysohn space $\Ur$, a space to which it is closely related. The story of $\s$ is quite uncommon: like $\Ur$, it was constructed in the late twenties by Urysohn (hence quite early in the history of metric geometry) but was completely forgotten for a long time. It is only recently that it was brought back on the research scene, thanks in particular to the work of Kat\v{e}tov \cite{K} which was quickly followed by several results due to Uspenskij \cite{Us1}, \cite{Us2} and later supported by several contributions by Vershik \cite{Ve1}, \cite{Ve2}, Gromov \cite{G}, Pestov \cite{Pe0} and Bogatyi \cite{B1}, \cite{B2}. Today, the spaces $\Ur$ and $\s$ are objects of active research and are studied by many different authors under many different aspects, see \cite{Pr}. 

Apart from the fact that both $\mathbb{S} ^{\infty}$ and $\s$ are complete, separable and ultrahomogeneous, the study of $\s$ is believed to be relevant for the distortion problem for $\ell _2$ because, from a Ramsey-theoretic point of view, the spaces $\mathbb{S}^{\infty}$ and $\s$ behave in a very similar way. For example, the following analog of Milman's theorem holds for $\s$: 

\begin{thm*}[Pestov \cite{Pe0}]

Let $\gamma$ be a finite partition of $\s$. Then for every
$\varepsilon > 0$ and every compact $K \subset \s$, there is $\Gamma \in \gamma$ and an isometric copy
$\widetilde{K}$ of $K$ in $\s$ such that
$\widetilde{K} \subset (\Gamma)_{\varepsilon}$.

\end{thm*}

In fact, since the work of Gromov and Milman \cite{GM} and of Pestov \cite{Pe0}, it is known that this analogy is only the most elementary form of a very general Ramsey-theoretic theorem. It is also known that this latter result has a very elegant reformulation at the level of the surjective isometry groups $\iso (\mathbb{S}^{\infty})$ and $\iso (\s)$ (seen as topological groups when equipped with the pointwise convergence topology). Call a topological group $G$ \emph{extremely amenable} when every continuous action of $G$ on a compact space admits a fixed point. Then on the one hand:

\begin{thm*}[Gromov-Milman \cite{GM}]  

The group $\iso (\mathbb{S}^{\infty})$ is extremely amenable. 

\end{thm*}

While on the other hand: 

\begin{thm*}[Pestov \cite{Pe0}]  

The group $\iso (\s)$ is extremely amenable. 

\end{thm*}

Actually, even more is known as both $\iso (\mathbb{S}^{\infty})$ and $\iso (\s)$ are known to satisfy the so-called \emph{L\'evy property} (cf Gromov-Milman \cite{GM} for $\iso (\mathbb{S}^{\infty})$ and Pestov \cite{Pe2} for $\iso (\s)$), a property shown to be stronger than extreme amenability by Gromov and Milman in \cite{GM}. 

In this note, we prove that:

\begin{thm}  

\label{thm:s mos}

The Urysohn sphere $\s$ is approximately indivisible. 

\end{thm}

In other words, for every finite partition $\gamma$ of $\s$ and  $\varepsilon > 0$, there is $\Gamma \in \gamma$ such that $(\Gamma)_{\varepsilon}$ includes an isometric copy of $\s$. Or equivalently, in terms of oscillation stability for topological groups, the standard action of $\iso (\s)$ on $\s$ is oscillation stable. Theorem \ref{thm:s mos} therefore exhibits an essential Ramsey-theoretic distinction between $\mathbb{S}^{\infty}$ and $\s$. At the level of $\iso (\mathbb{S}^{\infty})$ and $\iso (\s)$, it answers a question mentioned by Kechris, Pestov and Todorcevic in \cite{KPT}, Hjorth in \cite{Hj} and Pestov in \cite{Pe1}, and highlights a deep topological difference which, for the reasons mentioned previously, was not at all apparent until now. 

Our proof of Theorem \ref{thm:s mos} is combinatorial and rests on a discretization method largely inspired from the proof by Gowers of the stabilization theorem for the unit sphere $\mathbb{S}_{c_0}$ of $c_0$ and its positive part $\mathbb{S}_{c_0} ^+$. Recall that $c_0$ is the space of all real sequences converging to $0$ equipped with the $\left\| \cdot \right\|_{\infty}$ norm, and that $\mathbb{S}_{c_0} ^+$ is the set of all those elements of $\mathbb{S}_{c_0}$ taking only positive values. In \cite{Go}, Gowers studied the indivisibility properties of the spaces $\mathrm{FIN}_m$ (resp. $\mathrm{FIN}_m ^+$) of all the elements of $\mathbb{S}_{c_0}$ taking only values in $\{ k/m : k \in [-m,m] \cap \mathbb{Z} \}$ (resp. $\{ k/m : k \in \{0, 1,\ldots , m \} \}$) where $m$ ranges over the strictly positive integers:

\begin{thm*}[Gowers \cite{Go}] 

Let $m \in \N$, $m \geqslant 1$. Then $\mathrm{FIN}_m$ (resp. $\mathrm{FIN}_m ^+$) is $1$-indivisible (resp. indivisible).

\end{thm*}

A strong form of these results (see \cite{Go} for the precise statement) then led to:

\begin{thm*}[Gowers \cite{Go}] 
The sphere $\mathbb{S}_{c_0}$ (resp. $\mathbb{S}_{c_0} ^+$) is approximately indivisible. 
\end{thm*}

Here, our proof builds on the following discretization result proved in \cite{LANVT} and involving a family $(\Ur _m)_{m \geqslant 1}$ of
countable metric spaces. For $m\geqslant 1$, the space $\Ur _m$ is defined as follows: up to isometry it is the unique countable ultrahomogeneous metric space with distances in $\{1, \ldots, m \}$ into which every countable metric space with distances in $\{1, \ldots, m \}$ embeds isometrically. Then:

\begin{thm*}[Lopez-Abad - Nguyen Van Th\'e \cite{LANVT}]

\label{thm:TFAE S mos1}

The following are equivalent:

\begin{enumerate}
\item The space $\s$ is approximately indivisible.
\item For every strictly positive $m \in \N$, $\Ur _m$ is indivisible.
\end{enumerate}
\end{thm*}

The main ideas of the implication $(ii) \rightarrow (i)$ are presented for completeness in section \ref{section:appendix} together with an explanation as of why the spaces $\Ur _m$ are relevant as well as why some of the previous attempts to prove Theorem \ref{thm:s mos} failed. For more details, see the original reference \cite{LANVT} or \cite{NVT}. In the present paper, we show: 

\begin{thm} 

\label{thm:U_m indiv}

Let $m \in \N$, $m \geqslant 1$. Then $\Ur _m$ is indivisible. 

\end{thm}
 
Theorem \ref{thm:U_m indiv} expands the list of already known partition results of so-called countable ultrahomogeneous relational structures. Those structures appeared in the late fifties thanks to the pioneering work of Fra\"iss\'e \cite{Fr} and have since been studied from various points of view. This led in particular to several deep combinatorial classification results (see Lachlan-Woodrow \cite{LW} for graphs, Schmerl \cite{Sch} for partially ordered sets, Cherlin \cite{Ch} for directed graphs, or more recently Gray-Macpherson \cite{GMc} for connected-ultrahomogeneous graphs) but also to substantial developments in permutation group theory (e.g. Cameron \cite{Cam}, Truss \cite{Tr}), logic (e.g. Pouzet-Roux \cite{Pou}), or Ramsey theory (initiated by Komj\'ath-R\"odl \cite{KR}). However, although our paper really belongs to combinatorics, several consequences of Theorem \ref{thm:s mos} related to functional analysis deserve to be mentioned. They are based on the following fact, which is easily seen to be equivalent to Theorem \ref{thm:s mos}: 

\begin{thm}

\label{thm:univ indiv}

Let $\m{X}$ be a separable metric space with finite diameter $\delta$. Assume that every separable metric space with diameter less or equal to $\delta$ embeds isometrically into $\m{X}$. Then $\m{X}$ is approximately indivisible.  

\end{thm} 

When applied to the unit sphere of certain remarkable Banach spaces, this theorem yields interesting consequences. For example, it is known that every separable metric space with diameter less or equal to $2$ embeds isometrically into the unit sphere $\mathbb{S}_{\mathcal{C}[0,1]}$ of the Banach space $\mathcal{C}[0,1]$. It follows that:

\begin{thm}

The unit sphere of $\mathcal{C}[0,1]$ is approximately indivisible. 

\end{thm}

On the other hand, it is also known that $\mathcal{C}[0,1]$ is not the only space having a unit sphere satisfying the hypotheses of Theorem \ref{thm:univ indiv}. For example, Holmes proved in \cite{H} that there is a Banach space $\m{H}$ such that for every isometry $i : \funct{\Ur}{\m{Y}}$ of the Urysohn space $\Ur$ into a Banach space $\m{Y}$ with $0_{\m{Y}}$ is in the range of $i$, there is an isometric isomorphism between $\m{H}$ and the closed linear span of $i\left(\Ur\right)$ in $\m{Y}$. Very little is known about the space $\m{H}$, but it is easy to see that its unit sphere contains isometrically every separable metric space with diameter less or equal to $2$. Therefore:  

\begin{thm}

The unit sphere of the Holmes space is approximately indivisible. 

\end{thm}

Observe that these result do \emph{not} say that for $\m{X} = \mathcal{C}[0,1]$ or $\m{H}$, every finite partition $\gamma$ of the unit sphere $\mathbb{S}_{\m{X}}$ of $\m{X}$ and every $\varepsilon > 0$, there is $\Gamma \in \gamma$ and a closed infinite dimensional subspace $\m{Y}$ of $\m{X}$ such that $\mathbb{S}_{\m{X}} \cap \m{Y} \subset (\Gamma)_{\varepsilon}$: according to the classical results about oscillation stability in Banach spaces, this latter fact is false for those Banach spaces into which every separable Banach space embeds linearly, and it is known that both $\mathcal{C}[0,1]$ and $\m{H}$ have this property.

The paper is organized as follows. Section \ref{section:partition} corresponds to a short presentation of the partition theory of countable ultrahomogeneous structures with free amalgamation. In section \ref{section:notations}, the essential ingredients, the main technical results (Lemma \ref{thm:orbit} and Lemma \ref{lem:psi}) as well as the general outline of the proof of Theorem \ref{thm:U_m indiv} are presented. Finally, the proof of Lemma \ref{thm:orbit} is presented in section \ref{section:indiv}, while section \ref{section:appendix} presents a brief history of the problem of approximate indivisibility of $\s$ together with an outline of the proof of the aforementioned result of Lopez-Abad and the first author. 

\

\textbf{Acknowledgements.} L. Nguyen Van Th\'e would like to acknowledge the support of the Department of Mathematics \& Statistics Postdoctoral Program at the University of Calgary. N. W. Sauer was supported by NSERC of Canada Grant \# 691325. We would also like to thank Jordi Lopez-Abad, Vitali Milman, Stevo Todorcevic, the members of the Equipe de Logique set theory group at the University of Paris 7 and the anonymous referee for the considerable improvements their helpful comments and suggestions brought to the paper.

\section{Partition theory of countable ultrahomogeneous structures with free amalgamation.}

\label{section:partition}

In this section, we present a brief history of the general theory of indivisibility of countable ultrahomogeneous relational structures. For the undefined notions and for a general introduction to the partition 
theory of countable ultrahomogeneous structures see \cite{S2}. As mentioned in the introduction, partition theory is one of the aspects under which countable ultrahomogeneous relational structures were traditionally studied. The paper \cite{KR} quickly followed by \cite{EZS1} initiates a series devoted to this field, and more precisely devoted to vertex partition results of countable ultrahomogeneous structures with free amalgamation (The partition theory for sets of substructures other than vertices is much more complicated, see \cite{LSV} and \cite{S5}). In \cite{EZS2} it is proven that if a countable ultrahomogeneous structure is indivisible then the stabilizers of finite subsets form a chain, which in the binary case is a chain under embedding. This then led to \cite{S3} in which it is shown, in the case of directed graphs, that if the partial order of the stabilizers is finite then the Ramsey degree is equal to the size of 
its maximal antichain. In \cite{EZS3} the finiteness condition was removed in the case that the partial order is a chain. \cite{S4} contains
the most general result from which it follows that the Ramsey degree of  a binary countable ultrahomogeneous structure with free 
amalgamation  is equal to the size of the maximal antichain of the partial order of finite set stabilizers under
embedding if this partial order is finite. Hence if this partial order is a chain then the ultrahomogeneous structure is 
indivisible. 

For metric spaces, this global theory does not apply as amalgamation is in general not free. Still, it allows to capture the most elementary cases and manages to give a hint of what the general result should be. Indeed, it is easy to see that the partial order of stabilizers of finite subsets forms a chain under isometric embedding. Moreover, if $m\leqslant 3$, then it can be noticed that $\Ur _m$ has free amalgamation. Hence if $m\leqslant 3$ then $\Ur _m$ is indivisible, a result which allowed to prove that $\s$ is $1/6$-indivisible in \cite{LANVT}. 

However, if $m>3$, then the situation changes drastically and requires the introduction of essentially new arguments to prove that the metric spaces $\Ur _m$ are indivisible. The presentation of those arguments is the purpose of the present paper.

\section{Notations and definitions.}

\label{section:notations}

In this section, we present the notions and objects that will play a central role throughout the paper. 

\subsection{Kat\v{e}tov maps and orbits.} \label{subsection:katetov maps and orbits} Given a metric space $\m{X} = (X, d^{\m{X}})$, a map $f~:~\funct{X}{]0,+\infty[}$ is \emph{Kat\v{e}tov over $\m{X}$} when \[\forall x, y \in X, \ \ |f(x) - f(y)| \leqslant d^{\m{X}} (x,y) \leqslant f(x) + f(y).\] 

Equivalently, one can extend the metric $d^{\m{X}}$ to $X \cup \{ f \}$ by defining, for every $x, y$ in $X$, $ \widehat{d^{\m{X}}} (x, f) = f(x)$ and $\widehat{d^{\m{X}}} (x, y) = d^{\m{X}} (x, y)$. The corresponding metric space is then written $\m{X} \cup \{ f
\}$. The set of all Kat\v{e}tov maps over $\m{X}$ is written $E(\m{X})$. For a metric subspace $\m{X}$ of $\m{Y}$, a Kat\v{e}tov map $f \in E(\m{X})$ and a point $y \in \m{Y}$, then $y$ \emph{realizes $f$ over $\m{X}$} if \[ \forall x \in \m{X} \ \ d^{\m{Y}}(y,x) = f(x). \] 

The set of all $y \in \m{Y}$ realizing $f$ over $\m{X}$ is then written $O(f,\m{Y})$ and is called the \emph{orbit of $f$ in $\m{Y}$}. When $\m{Y}$ is implied by  context, the set $O(f,\m{Y})$ is simply written $O(f)$. Here, the concepts of Kat\v{e}tov map and orbit are relevant because of the following standard reformulation of the notion of ultrahomogeneity, which will be used extensively in the sequel: 

\begin{lemma}
\label{prop:extension} Let $\m{X}$ be a countable metric space. Then $\m{X}$ is ultrahomogeneous
iff for every finite subspace  $\m{F} \subset \m{X}$ and every Kat\v{e}tov map $f$ over $\m{F}$, if
$\m{F} \cup \{ f \}$ embeds into $\m{X}$, then $O(f, \m{X}) \neq \emptyset$.
\end{lemma}

For a proof of that fact in the general context of relational structures, see for example \cite{Fr}. For a proof in the particular context of metric spaces, see \cite{NVT}.

Throughout the paper, we will extensively use the result of Lemma \ref{prop:extension} when $\m{X} = \Ur _p$, where $p \geqslant 1$ is an integer. Recall that the space $\Ur _p$ is defined as follows: it is a countable, ultrahomogeneous metric space with distances in $\{1, \ldots, p\}$, and every countable metric space with distances in $\{1, \ldots, p \}$ embeds isometrically. Moreover, it can be proved that any two countable ultrahomogeneous metric spaces with the same finite metric subspaces are isometric (again, this is a standard fact in the context of countable ultrahomogeneous relational structures, see \cite{Fr} for a general proof or \cite{NVT} for the case of metric spaces). Therefore, the aforementioned properties completely characterize $\Ur _p$ up to isometry.  

There are several ways to look at $\Ur _p$. For example, it might be seen as a very simplified version (as Vladimir Pestov would say, a "poor man's version") of the Urysohn space $\Ur$ already mentioned in the introduction. The space $\Ur$ is, up to isometry, the unique complete separable ultrahomogeneous metric space that is also \emph{universal} for the class of all separable metric spaces (into which any separable metric space embeds isometrically). The space $\Ur$ was constructed by Urysohn in \cite{U} whose goal was precisely to prove the existence of a universal separable metric space, and there are nowadays several known characterizations and constructions of $\Ur$. For more information about it and its corresponding recent research developments, the reader should refer to the volume \cite{Pr}. In the present article however, it is more important to think of the space $\Ur _p$ as a discretized version of the Urysohn sphere $\s$ (after having replaced the metric $d^{\Ur _p}$ by $d^{\Ur _p}/p$) whose indivisibility properties capture the oscillation stability of $\s$, see \cite{LANVT} for the details, or section \ref{section:appendix} of the present paper for the main ideas.  

\

\textbf{Remark:} The notion of  Kat\v{e}tov function has become standard in the   Urysohn space literature because of the construction of $\Ur$ by Kat\v{e}tov in \cite{K}, often considered as the starting point of the present research about $\Ur$. They appeared prominently in several very different contributions to the field, see for example Cameron-Vershik \cite{CV}, Melleray \cite{Me1}, Pestov \cite{Pe0}, Uspenskij \cite{Us2}, or Vershik \cite{Ve2}. However, the idea of Kat\v{e}tov function already appears in the original article \cite{U} by Urysohn and is undoubtedly in the spirit of the constructions provided by Fra\"iss\'e in \cite{Fr}. Very likely, as examplified by the referee or by Maurice Pouzet, we are unaware of many other uses of those objects made by other authors, e.g. Isbell \cite{I} or Flood \cite{Fl1}, \cite{Fl2}.

\

\subsection{A notion of largeness.} 

In this section, $p$ is a fixed strictly positive integer. For metric spaces $\m{X}$, $\m{Y}$ and $\m{Z}$, write $\m{X} \cong \m{Y}$ if there is an isometry from $\m{X}$ onto $\m{Y}$ and define the set $\binom{\m{Z}}{\m{X}}$ as 
\[ \binom{\m{Z}}{\m{X}} = \{ \mc{X} \subset \m{Z} : \mc{X} \cong \m{X} \}. \] 

\begin{defn}

\label{defn:P}

The set $\mathbb{P}$ is the set of all ordered pairs of the form $s = (f_s,\m{C}_s)$ where 
\begin{enumerate}
\item $\m{C}_s \in \binom{\Ur _p}{\Ur _p}$.
\item $f_s$ is a map with finite domain $\dom f_s \subset \m{C}_s$ and with values in $\{ 1,\ldots , p\}$.
\item $f_s \in E(\dom f_s)$, ie $f_s$ is Kat\v{e}tov on its domain.
\end{enumerate}
The set $\mathbb{P}$ is partially ordered by the relation $\leqslant$ defined by \[ \forall s, t \in \mathbb{P} \ \ t \leqslant s \leftrightarrow \left(\dom f_s \subset \dom f_t \subset \m{C}_t \subset \m{C}_s \ \ \mathrm{and}  \ \ \restrict{f_t}{\dom f_s} = f_s\right). \] 
Finally, if $k \in \N$, then $t \leqslant _k s $ stands for \[ t \leqslant s \ \ \mathrm{and} \ \ \min f_t = \left \{ \begin{array}{cl} 
 \min f_s - k & \textrm{if $\min f_s > k$,} \\
 1 & \textrm{otherwise.}
 \end{array} \right.\]

\end{defn}

Observe that if $s \in \mathbb{P}$, then the ultrahomogeneity of $\Ur _p$ ensures that the set $O(f_s , \m{C} _s)$ is not empty and isometric to $\Ur _n$ where $n = \min (2\min f_s , p)$ (indeed, $O(f_s , \m{C} _s)$ is countable ultrahomogeneous with distances in $\{1,\ldots ,n \}$ and embeds every countable metric space with distances in $\{1,\ldots ,n \}$). Observe also that there is always a $t \in \mathbb{P}$ such that $t \leqslant _1 s$. Observe finally that unlike the relations $\leqslant$ and $\leqslant _0$, the relation $\leqslant _k$ is not transitive when $k>0$.

\begin{defn}

\label{defn:psi}

Let $s \in \mathbb{P}$ and $\Gamma \subset \Ur _p$. The notion of \emph{largeness of $\Gamma$ relative to $s$} is defined recursively as follows:

If $\min f_s = 1$, then $\Gamma$ is large relative to $s$ iff \[ \forall t \leqslant _0 s  \ \left(O(f_t , \m{C}_t) \cap  \Gamma \ \textrm{is infinite}\right).\]

If $\min f_s > 1$, then $\Gamma$ is large relative to $s$ iff \[ \forall t \leqslant _0 s  \ \ \exists u \leqslant _1 t \ \ \left(\textrm{$\Gamma$ is large relative to $u$}\right) .\]

\end{defn}

The idea behind the definition of largeness is that if $\Gamma$ is large relative to $s$, then inside $\m{C}_s$ the set $\Gamma$ should represent a substantial part of the orbit of $f_s$. This intuition is made precise by the following Lemma: 

\begin{lemma}

\label{thm:orbit}

Let $s \in \mathbb{P}$. Assume that $\Gamma$ is large relative to $s$. Then there exists an isometric copy $\m{C}$ of $\Ur _p$ inside $\m{C}_s$ such that:
\begin{enumerate} 
\item $\dom f_s \subset \m{C}$.
\item $O(f_s,\m{C}) \subset \Gamma$.  
\end{enumerate}  
\end{lemma}

In words, Lemma \ref{thm:orbit} means that by thinning up $\m{C}_s$, it is possible to ensure that the whole orbit of $f_s$ is included in $\Gamma$. The requirement $\dom f_s \subset \m{C}$ guarantees that the orbit of $f_s$ in the new space has the same metric structure as the orbit of $f_s$ in the original space. The proof of Lemma \ref{thm:orbit} represents the core of the proof of Theorem \ref{thm:U_m indiv} and is detailed in section \ref{section:indiv}. The second crucial fact about $\mathbb{P}$ and largeness lies in:  

\begin{lemma}

\label{lem:psi}

Let $s \in \mathbb{P}$ be such that $\Gamma$ is not large relative to $s$. Then there is $t \leqslant _0 s$ such that $\Ur _p \smallsetminus \Gamma$ is large relative to $t$.  
\end{lemma}

\begin{proof}

We proceed by induction on $\min f_s$. If $\min f_s = 1$, then there is $t \leqslant _0 s $ such that \[ O(f_t , \m{C}_t) \cap \Gamma \ \ \textrm{is finite}.\]

It is then clear that $\Ur _p \smallsetminus \Gamma$ is large relative to $t$. On the other hand, if $\min f_s > 1$, then there is $t \leqslant _0 s$ such that \[ \forall w \leqslant _1 t \ \ \textrm{$\Gamma$ is not large relative to $w$}.\] 

We claim that $\Ur _p \smallsetminus \Gamma$ is large relative to $t$: let $u \leqslant _0 t$. We want to find $v \leqslant _1 u $ such that $\Ur _p \smallsetminus \Gamma$ is large relative to $v$. Let $w \leqslant _1 u$. Then $w \leqslant _1 t$ and it follows that $\Gamma$ is not large relative to $w$. By induction hypothesis, since $\min f_w < \min f_u = \min f_t$ there is $v \leqslant _0 w $ such that $\Ur _p \smallsetminus \Gamma$ is large relative to $v$. Additionally $v \leqslant _1 u$. Thus $v$ is as required. \end{proof}

When combined, Lemma \ref{thm:orbit} and Lemma \ref{lem:psi} lead to Theorem \ref{thm:U_m indiv} as follows: Take $p=m$ and consider a finite partition $\gamma$ of $\Ur_m$. Without loss of generality, $\gamma$ has only two parts, namely $\Pi$ (purple points) and $\Omega$ (orange points). Fix $t \in \mathbb{P}$ such that $\min f _t = m$. According to Lemma \ref{lem:psi}, either $\Pi$ is large relative to $t$ or there is $u \leqslant _0 s$ such that $\Omega$ is large relative to $u$. In any case, there are $s \in \{t, u\}$ and $\Gamma \in \{\Pi,\Omega\}$ such that $\min f_s = m$ and $\Gamma$ is large relative to $s$. Applying Lemma \ref{thm:orbit} to $s$, we obtain a copy $\m{C}$ of $\Ur _m$ inside $\m{C}_s$ such that $\dom f_s \subset \m{C}$ and $O(f_s , \m{C}) \subset \Gamma$. Observe that $O(f_s , \m{C})$ is isometric to $\Ur _m$. $\qed$ 

\

The remaining part of this article is therefore devoted to a proof of Lemma \ref{thm:orbit}.

\section{Proof of Lemma \ref{thm:orbit}.}

\label{section:indiv}

From now on, the integer $p > 0$ is fixed together with $\Gamma \subset \Ur _p$. We proceed by induction and prove that for every strictly positive $m \in \N$ with $m \leqslant p$ the following statement $\mathcal{J}_m$ holds: 

\
$\mathcal{J} _m$ : ''For every $s \in \mathbb{P}$ such that $\min f_s = m$, if $\Gamma$ is large relative to $s$, then there exists an isometric copy $\m{C}$ of $\Ur _p$ inside $\m{C}_s$ such that:
\begin{enumerate} 
\item $\dom f_s \subset \m{C}$.
\item $O(f_s , \m{C}) \subset \Gamma$.''  
\end{enumerate} 

This section is organized as follows. In subsection \ref{subsection:reformulation}, we show that the statement $\mathcal{J}_m$ is equivalent to a stronger statement denoted $\mathcal{H}_m$. This is achieved thanks to a technical lemma (Lemma \ref{lem:red}) about the structure of the orbits in $\Ur _p$ and whose proof is postponed to subsection \ref{subsection:red}. In subsection \ref{subsection:J_1}, we initiate the proof by induction and show that the statement $\mathcal{J}_1$ holds. We then show that if $\mathcal{H}_j$ holds for every $j<m$, then $\mathcal{J}_m$ holds. The general strategy of the induction step is presented in subsection \ref{subsection:induction}, while \ref{subsection:sequences} provides the details for the most technical aspects.  

\subsection{Reformulation of $\mathcal{J}_m$.}

\label{subsection:reformulation}

As mentioned previously, we start by reformulating the statement $\mathcal{J}_m$ under a form which will be useful when performing the induction step. For a function $f$ and a subset $F$ of the domain $\dom f$ of $f$, we write $\restrict{f}{F}$ for the restriction of $f$ to $F$. Consider the following statement, denoted $\mathcal{H}_m$: 

$\mathcal{H} _m$ : ''For every $s \in \mathbb{P}$ and every $F \subset \dom f_s$ such that $\min \restrict{f_s}{F} = \min f_s = m$, if $\Gamma$ is large relative to $s$, then there exists an isometric copy $\m{C}$ of $\Ur _p$ inside $\m{C}_s$ such that:
\begin{enumerate} 
\item $\dom f_s \cap \m{C} = F$.
\item $O(\restrict{f_s}{F} , \m{C}) \subset \Gamma$.''  
\end{enumerate}

The statement $\mathcal{J}_m$ is clearly implied by $\mathcal{H}_m$: simply take $F = \dom f_s$. The purpose of the following lemma is to show that the converse is also true. 

\begin{lemma}

\label{lem:JmHm}

The statement $\mathcal{J}_m$ implies the statement $\mathcal{H}_m$.  

\end{lemma}

\begin{proof}
Our main tool here is the following technical result, whose proof is postponed to section \ref{subsection:red}. 

\begin{lemma}

\label{lem:red}

Let $G_0 \subset G$ be finite subsets of $\Ur_p$, $\mathcal{G}$ a family of Kat\v{e}tov maps with domain $G$ and such that for every $g, g' \in \mathcal{G}$: \[ \max(\restrict{|g - g'|}{G_0}) = \max | g - g'|, \] \[ \min(\restrict{(g+g')}{G_0}) = \min(g + g').\] 

Then there exists an isometric copy $\m{C}$ of $\Ur _p$ inside $\Ur _p$ such that:
\begin{enumerate} 
\item $G \cap \m{C} = G_0$.
\item $\forall g \in \mathcal{G} \ \ O(\restrict{g}{G_0} , \m{C}) \subset O(g, \Ur _p).$  
\end{enumerate} 

\end{lemma} 

Note that under the conditions of Lemma \ref{lem:red}, the restriction map $g \mapsto \restrict{g}{G_0}$ is one-to-one. Assuming Lemma \ref{lem:red}, here is how $\mathcal{J}_m$ implies $\mathcal{H}_m$: let $s$ and $F$ be as in the hypothesis of $\mathcal{H}_m$. Apply $\mathcal{J}_m$ to $s$ to get an isometric copy $\mc{C}$ of $\Ur _p$ inside $\m{C}_s$ such that $\dom f_s \subset \mc{C}$ and $O(f_s , \mc{C}) \subset \Gamma$. Apply then Lemma \ref{lem:red} inside $\mc{C}$ to $F \subset \dom f_s$ and the family $\{ f_s\}$ to get an isometric copy $\m{C}$ of $\Ur _p$ inside $\mc{C}$ such that $\dom f_s \cap \m{C} = F$ and $O(\restrict{f_s}{F} , \m{C}) \subset O(f_s, \mc{C})$. Then $\m{C}$ is as required. \end{proof}

\subsection{Proof of $\mathcal{J}_1$.} 

\label{subsection:J_1}

Consider an enumeration $\{ x_n : n \in \N\}$ of $\m{C}_s$ admitting $\dom f_s$ as an initial segment. Assume that the points $\varphi (x_0),\ldots , \varphi(x_n)$ are constructed so that: \begin{itemize}
	\item The map $\varphi$ is an isometry.
	\item $\restrict{\varphi}{\dom f_s} = id_{\dom f_s}$.
	\item $\varphi (x_k) \in \Gamma$ whenever $\varphi(x_k)$ realizes $f_s$ over $\dom f_s$. 
\end{itemize} 

We want to construct $\varphi (x _{n+1})$. Consider $h$ defined on $\{ \varphi (x_k) : k \leqslant n\} $ by: 
\[ \forall k \leqslant n \ \ h(\varphi(x_k)) = d^{\m{C}_s} (x_k , x_{n+1}).\] 

Observe that the metric subspace of $\m{C}_s$ given by $\{x_k : k \leqslant n+1\}$ witnesses that $h$ is Kat\v{e}tov. It follows that the set of all $y \in \m{C}_s \smallsetminus \dom f_s$ realizing $h$ over $\{ \varphi (x_k) : k \leqslant n\}$ is not empty and $\varphi (x_{n+1})$ can be chosen in that set. Additionally, observe that if $\restrict{h}{\dom f_s} = f_s$, then the fact that $\min f_s = 1$ and $\Gamma$ is large relative to $s$ then guarantees that $h$ can be realized by a point in $\Gamma$. We can therefore choose $\varphi (x_{n+1})$ to be one of those points. After infinitely many steps, the subspace $\m{C}$ of $\m{C}_s$ supported by $\{ \varphi (x_n) : n \in \N\}$ is as required. $\qed$

\subsection{Induction step.}

\label{subsection:induction}

Assume that the statements $\mathcal{J}_1\ldots\mathcal{J}_{m-1}$, and therefore the statements $\mathcal{H}_1\ldots\mathcal{H}_{m-1}$ hold. We are going to show that $\mathcal{J}_m$ holds. So let $s \in \mathbb{P}$ such that $\min f_s = m$ and $\Gamma$ is large relative to $s$. To make the notation easier, we assume that $s$ is of the form $(f,\Ur _p)$ and we write $F$ instead of $\dom f$. We need to produce an isometric copy $\m{C}$ of $\Ur _p$ inside $\Ur _p$ such that $F \subset \m{C}$ and $O(f , \m{C}) \subset \Gamma$. This is achieved inductively thanks to the following lemma. Recall that for metric subspaces $\m{X}$ and $\m{Y}$ of $\Ur _p$ and $\varepsilon > 0$, the sets $(\m{X}) _{\varepsilon}$ and $\binom{\m{Y}}{\Ur _p}$ are defined by: 
\[(\m{X}) _{\varepsilon} = \{ y \in \Ur _p : \exists x \in \m{X} \ \ d ^{\Ur _p} (y,x) \leqslant \varepsilon \},\]
\[ \binom{\m{Y}}{\Ur _p} = \{ \mc{U} \subset \m{Y} : \mc{U} \cong \Ur _p \}. \]

\begin{lemma}

\label{lem:ind1}

Let $\m{X}$ be a finite subspace of $\Ur _p$ and $\m{A} \in \binom{\Ur _p}{ \Ur _p}$ such that: 

\vspace{0.5em}
\hspace{1em}
(i) $F \subset \m{X} \subset \m{A}$.

\vspace{0.5em}
\hspace{1em}
(ii) $\left( \m{X} \right)_{m-1} \cap O(f,\m{A}) \subset \Gamma$.

\vspace{0.5em}
\hspace{1em}
(iii) $\forall g \in E(\m{X}) \ \ \left(\restrict{g}{F} = \restrict{f}{F}\right) \rightarrow \left(\textrm{$\Gamma$ is large relative to $(g, \m{A})$}\right)$.

\vspace{0.5em}

Then for every $h \in E(\m{X})$, there are $\m{B} \in \binom{\m{A}}{\Ur _p}$ and $x^* \in \m{B}$ realizing $h$ over $\m{X}$ such that: 

\vspace{0.5em}
\hspace{1em}
(i') $F \subset (\m{X} \cup \{x^*\}) \subset \m{B}$.

\vspace{0.5em}
\hspace{1em}
(ii') $ \left(\m{X} \cup \{ x^*\} \right)_{m-1} \cap O(f, \m{B}) \subset \Gamma$.

\vspace{0.5em}
\hspace{1em}
(iii') $\forall g \in E(\m{X} \cup \{ x^*\}) \ \ \left(\restrict{g}{F} = \restrict{f}{F}\right) \rightarrow \left(\textrm{$\Gamma$ is large relative to $(g, \m{B})$}\right)$.
\end{lemma}

\begin{claimm}
Lemma \ref{lem:ind1} implies $\mathcal{J} _m$. 
\end{claimm}

\begin{proof}
The required copy of $\m{C}$ can be constructed inductively. We start by fixing an enumeration $\{ x_n : n \in \N \}$ of $\Ur _p$ such that $F = \{x_0 ,\ldots , x_k \}$ and by setting $\tilde{x}_i = x_i$ for every $i \leqslant k$. Next, we proceed as follows: set $\m{A}_k = \Ur _p$. Then the subspace of $\Ur _p$ supported by $\{\tilde{x}_0 ,\ldots , \tilde{x}_k\}$ and the copy $\m{A}_k$ satisfy the requirements (i)-(iii) of Lemma \ref{lem:ind1}. Consider then $h_{k+1}$ defined on $\{\tilde{x}_0 ,\ldots , \tilde{x}_k\}$ by: \[ \forall i \leqslant k \ \ h_{k+1}(\tilde{x}_i) = d^{\Ur _p}(x_{k+1} , x_i).\] 

Then $h_{k+1}$ is Kat\v{e}tov over $\{\tilde{x}_0 ,\ldots , \tilde{x}_k\}$ and Lemma \ref{lem:ind1} can be applied to the subspace of $\Ur _p$ supported by $\{\tilde{x}_0 ,\ldots , \tilde{x}_k\}$, the copy $\m{A}_k$ and the Kat\v{e}tov map $h_{k+1}$. It produces $x^*$ and $\m{B}$, and we set $\tilde{x} _{k+1} = x^*$ and $\m{A}_{k+1} =\m{B}$. In general, assume that $\tilde{x}_0 ,\ldots , \tilde{x}_l$ and $\m{A}_k ,\ldots , \m{A}_l$ are constructed so that $\m{A}_l$ and the subspace of $\Ur _p$ supported by $\{\tilde{x}_0 ,\ldots , \tilde{x}_l\}$ satisfy the hypotheses of Lemma \ref{lem:ind1}. Consider $h_{l+1}$ defined on $\{\tilde{x}_0 ,\ldots , \tilde{x}_l\}$ by: \[ \forall i \leqslant l \ \ h_{l+1}(\tilde{x}_i) = d^{\Ur _p}(x_{l+1} , x_i).\] 

Then $h_{l+1}$ is Kat\v{e}tov over $\{\tilde{x}_0 ,\ldots , \tilde{x}_l\}$, Lemma \ref{lem:ind1} can be applied to produce $x^*$ and $\m{B}$, and we set $\tilde{x} _{l+1} = x^*$ and $\m{A}_{l+1} = \m{B}$. After infinitely steps, we are left with $\m{C} = \{ \tilde{x}_n : n \in \N\}$ isometric to $\Ur _p$, as required. \end{proof} 

The remaining part of this section is consequently devoted to a proof of Lemma~\ref{lem:ind1} where $\m{X}$, $\m{A}$ and $h$ are fixed according to the requirements (i)-(iii) of Lemma \ref{lem:ind1}. 

\begin{claimm}
If $x^*$ and $\m{B}$ satisfy (i') and (ii') of Lemma~\ref{lem:ind1}, then (iii') is also satisfied. 
\end{claimm}

\begin{proof}
Let $g \in E( \m{X} \cup \{ x^*\})$ be such that $\restrict{g}{F} = \restrict{f}{F}$. We need to show that $\Gamma$ is large relative to $(g,\m{B})$. If $\min g \geqslant m$, then $(g,\m{B}) \leqslant _0 (f,\Ur _p)$. Since $\Gamma$ is large relative to $(f, \Ur _p)$, it follows that $\Gamma$ is also large relative to $(g,\m{B})$ and we are done. On the other hand, if $\min g \leqslant m-1$, then \[ O(g, \m{B}) \subset \left(\left(\m{X} \cup \{ x^*\} \right)_{m-1} \cap O(f, \m{B}) \right) \subset \Gamma.\] 

So $\Gamma$ is large relative to $(g,\m{B})$.\end{proof}

With this fact in mind, we define \[ K = \{ \phi \in E(\m{X} \cup \{ h \}) : \restrict{\phi}{F} = \restrict{f}{F} \ \ \mathrm{and} \ \ \phi(h) \leqslant m-1 \}.\] 

Two comments about notation before we go on. First, as specified in \ref{subsection:katetov maps and orbits}, $\m{X} \cup \{ h \}$ in the definition of $K$ above is understood as the one-point metric extension $\m{X} \cup \{ h \}$ of $\m{X}$ obtained by setting $ d (x, h) = h(x)$ and $d (x, y) = d^{\m{X}} (x, y)$ for every $x, y$ in $X$. Next, in the sequel, when $\m{X}\subset \Ur _p$ and $\m{X}\cup\{u\}, \m{X}\cup \{v\}$ are one-point metric extensions $\m{X}$ (provided by points of $\Ur _p \smallsetminus \m{X}$ or by Kat\v{e}tov maps over $\m{X}$), we will write $\m{X}\cup\{u\} \cong \m{X}\cup \{v\}$ when $d(x,u)=d(x,v)$ whenever $x\in \m{X}$. The reason for which $K$ is relevant here lies in the following claim. 

\begin{claimm}
Assume that $\m{B} \in \binom{\m{A}}{\Ur _p}$ and $x^* \in \m{B}$ are such that: 

\begin{enumerate}
\item $\m{X} \subset \m{B}$.
\item $x^*$ realizes $h$ over $\m{X}$.
\item For every $\phi \in K$, every point in $\m{B}$ realizing $\phi$ over $\m{X} \cup \{ x ^*\} \cong \m{X} \cup \{ h\}$ is in $\Gamma$. 
\item $x^* \in \Gamma$ if $\restrict{h}{F} = \restrict{f}{F}$ (that is if $x^* \in O(f, \m{B})$). 
\end{enumerate} 

Then $x^*$ and $\m{B}$ satisfy (i') and (ii') Lemma \ref{lem:ind1}.
\end{claimm}

\begin{proof}
The requirement (i') is obviously satisfied so we concentrate on (ii'). Let $y \in \left( \m{X} \cup \{ x^*\} \right)_{m-1} \cap O(f,\m{B})$. We need to prove that $y \in \Gamma$. If $y \in \left(\m{X} \right)_{m-1}$, then $y$ is actually in $\left(\m{X}\right)_{m-1} \cap O(f, \m{A}) \subset \Gamma$ and we are done. Otherwise, $y \in \left(\{x^* \}\right)_{m-1}$. If $y=x^*$, there is nothing to do: since $y$ is in $O(f, \m{B})$, so is $x^*$. Thus, by (iv), $x^* \in \Gamma$, that is $ y \in \Gamma$. Otherwise, let $\phi$ be the Kat\v{e}tov map realized by $y$ over $\m{X} \cup \{ x ^*\} \cong \m{X} \cup \{ h\}$. According to (iii), it suffices to show that $\phi \in K$. This is what we do now. First, the metric space $\m{X}\cup \{ x^*, y\}$ witnesses that $\phi$ is Kat\v{e}tov over $\m{X}\cup \{ h\}$. Next, $y \in O(f,\m{B})$ hence $\restrict{\phi}{F} = \restrict{f}{F}$. Finally, $\phi(h) = d^{\Ur _p} (x^*,y) \leqslant m-1$ since $y \in \left(\{x^* \}\right)_{m-1}$. \end{proof}

The strategy to construct $\m{B}$ and $x^*$ is the following one. Let $\{ \phi _{\alpha} : \alpha < \left|K\right|\}$ be an enumeration of $K$. We first construct a sequence of points $(x_{\alpha})_{\alpha < \left|K\right|}$ and a decreasing sequence $(\m{D}_{\alpha}) _{\alpha < \left|K\right|}$ of copies of $\Ur _p$ so that $x _{\alpha} \in \m{D}_{\alpha}$ and for every $\beta \leqslant \alpha < \left|K\right|$: 

\begin{enumerate}
\item $\m{X} \subset \m{D} _{\alpha}$. 
\item $x _{\alpha}$ realizes $h$ over $\m{X}$. 
\item Every point in $\m{D}_{\alpha}$ realizing $\phi _{\beta}$ over $\m{X} \cup \{ x_{\alpha}\} \cong \m{X} \cup \{ h\}$ is in $\Gamma$. 
\end{enumerate}

The details of this construction are provided in section \ref{subsection:sequences}. Once this is done, call $x'=x_{\left|K\right|-1}$, $\m{B}' = \m{D}_{\left|K\right|-1}$. The point $x'$ and the copy $\m{B}'$ are almost as required except that $x'$ may not be in $\Gamma$. If $\restrict{h}{F} \neq \restrict{f}{F}$, this is not a problem and setting $x^* = x'$ and $\m{B} = \m{B}'$ works. On the other hand, if $\restrict{h}{F} = \restrict{f}{F}$, then some extra work is required and we proceed as follows. 

Pick $x^* \in \m{B}'$ realizing $h$ over $\m{X}$ and such that $d^{\Ur _p}(x^*, x')=1$. We will be done if we construct $\m{B} \in \binom{\m{B}'}{\Ur _p}$ so that $\left(\m{X}  \cup \{ x^* , x' \}\right) \cap \m{B} = \m{X}  \cup \{ x^* \}$ and for every $\phi \in K$, every point in $\m{B}$ realizing $\phi $ over $\m{X} ^* \cup \{ x^* \}$ realizes $\phi $ over $\m{X} ^* \cup \{ x' \}$. Here is how this is achieved thanks to Lemma \ref{lem:red}. For $\phi \in K$, define the map $\hat{\phi}$ on $\m{X}  \cup \{ x^* , x' \}$ by
\begin{displaymath}
\left \{ \begin{array}{l}
 \restrict{\hat{\phi}}{\m{X}} = \restrict{\phi}{\m{X}}, \\
 \hat{\phi}(x^*) = \hat{\phi}(x') = \phi(h).
 \end{array} \right.
\end{displaymath} 

Using the fact that $\phi$ is Kat\v{e}tov over $\m{X}\cup \{h\}$ and $\m{X} \cup \{ x^*\} \cong \m{X} \cup \{ x'\} \cong \m{X} \cup \{ h\}$, it is easy to check that $\hat{\phi}$ is Kat\v{e}tov over $\m{X}  \cup \{ x^* , x' \}$ and that for every $\phi, \phi' \in K$: \[ \max(\restrict{|\hat{\phi} - \hat{\phi}'|}{\m{X} \cup \{ x^*\}}) = \max | \hat{\phi} - \hat{\phi}'|, \] \[ \min(\restrict{(\hat{\phi}+\hat{\phi}')}{\m{X} \cup \{ x^*\}}) = \min(\hat{\phi} + \hat{\phi}').\] 

Working inside $\m{B}'$, we can therefore apply Lemma \ref{lem:red} to $\m{X} \cup \{ x^*\} \subset \m{X}  \cup \{ x^* , x' \}$ and the family $(\hat{\phi})_{\phi \in K}$ to obtain $\m{B}$ as required. $\qed$

\subsection{Construction of the sequences $(x_{\alpha})_{\alpha < \left|K\right|}$ and $(\m{D}_{\alpha}) _{\alpha < \left|K\right|}$.}

\label{subsection:sequences}

The construction of the sequences $(x_{\alpha})_{\alpha < \left|K\right|}$ and $(\m{D}_{\alpha}) _{\alpha < \left|K\right|}$ is carried out thanks to a repeated application of the following lemma. Recall that the set $K$ of Kat\v{e}tov functions over $\m{X} \cup \{ h \}$ is defined by \[ K = \{ \phi \in E(\m{X} \cup \{ h \}) : \restrict{\phi}{F} = \restrict{f}{F} \ \ \mathrm{and} \ \ \phi(h) \leqslant m-1 \}.\] 

Note that when $\m{X} \cup \{ u \} \cong \m{X} \cup \{ h \}$, we will often see $K$ as a set of Kat\v{e}tov maps over $\m{X} \cup \{ u \}$. Every element of $K$ is then thought of as a Kat\v{e}tov map over $\m{X} \cup \{ u \}$ in the obvious manner.  

\begin{lemma}

\label{lem:ind2}

Let $\mathcal{F} \subset K$ and $\m{D} \in \binom{\m{A}}{\Ur _p}$ be such that $\m{X} \subset \m{D}$. Assume that $u \in \m{D}$ realizes $h$ over $\m{X}$ and is such that for every $\phi \in \mathcal{F}$, every point in $\m{D}$ realizing $\phi$ over $\m{X} \cup \{ u \} \cong \m{X} \cup \{ h \}$ is in $\Gamma$. Let $s \in K \smallsetminus \mathcal{F}$ be such that \[ \forall \phi \in K \ \ \phi (h) > s(h) \rightarrow \phi \in \mathcal{F} \ \ \textrm{and} \ \ \phi (h) < s(h) \rightarrow \phi \notin \mathcal{F}. \ \ \ (*)\] 

Then there are $\m{E} \in \binom{\m{D}}{\Ur _p}$ and $v \in \m{E}$ realizing $h$ over $\m{X} $ such that $\m{X} \subset \m{E}$ and for every $\phi \in \mathcal{F} \cup \{ s \}$, every point in $\m{E}$ realizing $\phi$ over $\m{X} \cup \{ v \} \cong \m{X} \cup \{ h \}$ is in $\Gamma$.
\end{lemma}

Once Lemma \ref{lem:ind2} is proven, here is how the sequences $(x_{\alpha})_{\alpha < \left|K\right|}$ and $(\m{D}_{\alpha}) _{\alpha < \left|K\right|}$ are constructed: choose the enumeration $\{ \phi _{\alpha} : \alpha < \left|K\right|\}$ of $K$ so that the sequence $(\phi _{\alpha} (h))_{\alpha < \left|K\right|}$ is nondecreasing. Apply Lemma \ref{lem:ind2} to $\mathcal{F} = \emptyset$, $\m{D} = \m{A}$ and $s = \phi_0$ to produce $x_0$ and $\m{D}_0$. In general, apply Lemma \ref{lem:ind2} to $\mathcal{F} = \{ \phi_0\ldots \phi_{\alpha} \}$, $\m{D} = \m{D}_{\alpha}$ and $s=\phi_{\alpha+1}$ to produce $x_{\alpha+1}$ and $\m{D}_{\alpha+1}$. After $|K|$ steps, the sequences $(x_{\alpha})_{\alpha < \left|K\right|}$ and $(\m{D}_{\alpha}) _{\alpha < \left|K\right|}$ are as required.

\begin{proof}[Proof of Lemma \ref{lem:ind2}]
We start with the case where $s(h) \geqslant \min \restrict{s}{\m{X}}$. The map $s$ being in $K$, $s(h) \leqslant m-1$ and so $\min \restrict{s}{\m{X}} \leqslant m-1$. Then, \[ O(\restrict{s}{\m{X}}, \m{D}) \subset \left( \left(\m{X}\right)_{m-1} \cap O(f, \m{D})\right).\] 

But from the requirement (ii) of Lemma \ref{lem:ind1}, \[ \left( \left(\m{X}\right)_{m-1} \cap O(f, \m{D})\right) \subset \Gamma.\] 

Observe now that every point in $\m{D}$ realizing $s$ over $\m{X} \cup \{ u \}$ is in $O(\restrict{s}{\m{X}}, \m{D})$. Thus, according to the previous inclusions, any such point is also in $\Gamma$. So in fact, there is nothing to do: $v = u$ and $\m{E} = \m{D}$ works. 

From now on, we consequently suppose that $s(h) < \min \restrict{s}{\m{X}}$. Let $s_1$ be defined on $\m{X} \cup \{ u \}$ by 
\begin{displaymath}
s_1(x) = \left \{ \begin{array}{cl}
 s(x) & \textrm{if $x \in \m{X}$,} \\
 s(h) + 1 & \textrm{if $x = u$.}
 \end{array} \right.
\end{displaymath}  

\begin{claimm}
The map $s_1$ is Kat\v{e}tov. 
\end{claimm}

\begin{proof}

The map $s$ is Kat\v{e}tov over $\m{X}$. Hence, it is enough to prove that for every $x \in \m{X}$, \[ \left| s_1(u) - s_1(x) \right| \leqslant d^{\Ur _p}(x,u) \leqslant s_1(u) + s_1(x). \] 

That is \[ \left| s(h) + 1 - s(x) \right| \leqslant h(x) \leqslant s(h) + 1 + s(x).\] 

Because $s$ is Kat\v{e}tov over $\m{X} \cup \{ h \}$, it is enough to prove that \[ s(h) + 1 - s(x) \leqslant h(x). \]

But this holds since $s(h) < \min \restrict{s}{\m{X}}$. \end{proof}

Note that, as pointed out by the referee, the previous claim also admits a nice geometric explanation: proving that $s_1$ is Kat\v{e}tov is equivalent to verifying that the metric space $\m{X} \cup \{u, s\}$ stays metric when the distance between $u$ and $s$ is increased by one. To do that, simply observe that any metric triangle with integer distances (in particular, here, those of the form $\{ x, u, s\}$) remains metric when a distance that is not the largest is increased by one (which is true here because $s(h) < \min \restrict{s}{\m{X}}$).

\begin{claimm}
$\Gamma$ is large relative to $(s_1, \m{D})$. 
\end{claimm}

\begin{proof}
If $s(h) = m-1$, then $\min s_1 = m = \min f$ and so $(s_1 , \m{D}) \leqslant _ 0 (f , \Ur _p)$. Since $\Gamma$ is large relative to $(f ,\Ur _p)$, it is also large relative to $(s_1 , \m{D})$ and we are done. On the other hand, if $s(h) < m-1$, then $s_1 \in K$ and it follows from the hypothesis $(*)$ on $\mathcal{F}$ that $s_1 \in \mathcal{F}$. In particular, every point in $\m{D}$ realizing $s_1 $ over $\m{X} \cup \{ u \}$ is in $\Gamma$, and it follows that $\Gamma$ is large relative to $(s_1 , \m{D})$. \end{proof}

Consequently, there is $(s_2 , \m{D} _{s_2}) \leqslant _1 (s_1 , \m{D})$ such that $\Gamma$ is large relative to $(s_2 , \m{D} _{s_2})$. We are now going to construct $v$ and a Kat\v{e}tov extension $s_3$ of $s_2$ such that $v$ realizes $h$ over $\m{X}$, $s_3 (v) = s(h)$ and $(s_3,\m{D}_{s_2}) \leqslant_0 (s_2 , \m{D}_{s_2})$. This last requirement will make sure that $\Gamma$ is large relative to $(s_3,\m{D}_{s_2})$. We will then apply Lemma \ref{lem:red} to obtain the copy $\m{E}$ as required. 
Here is how we proceed formally: fix $w \in O(s_2, \m{D}_{s_2})$ and consider the map $h_1$ defined on $\m{X} \cup \{ u, w \}$ by
\begin{displaymath}
h_1(x) = \left \{ \begin{array}{cl}
 h(x) & \textrm{if $x \in \m{X}$.} \\
 1 & \textrm{if $x = u$.} \\
 s(h) & \textrm{if $x = w$.}
 \end{array} \right.
\end{displaymath}

\begin{claimm}
The map $h _1$ is Kat\v{e}tov. 
\end{claimm}

\begin{proof}
The metric space $\left(\m{X} \cup \{ h\}\right) \cup \{ s\}$ witnesses that $\restrict{h_1}{\m{X} \cup \{ w \}}$ is Kat\v{e}tov. Next, $\restrict{h_1}{\m{X} \cup \{ u \}}$ is also Kat\v{e}tov: Let $x \in \m{X}$. Then \[ \left|h_1(x) - h_1(u) \right| = h(x) - 1 \leqslant h(x) = d^{\Ur _p} (x,u) \leqslant h(x) + 1 = h_1(x) + h_1(u).\] 

The only thing we still need to show is therefore \[ \left|h_1(u) - h_1(w) \right| \leqslant d^{\Ur _p} (u,w) \leqslant h_1(u) + h_1(w).\] 

But this inequalities hold as they are equivalent to \[ \left| 1 - s(h) \right| \leqslant s(h) + 1 \leqslant 1 + s(h).\qedhere\]  \end{proof}

Let $v \in \m{D} _{s_2}$ realizing $h_1$ over $\m{X} \cup \{ u, w\}$. As announced previously, define an extension $s _3$ of $s_2$ on $\dom s_2 \cup \{v\}$ by setting $s_3 (v) = s(h)$. 

\begin{claimm}
The map $s_3$ is Kat\v{e}tov and $\Gamma$ is large relative to $(s_3, \m{D}_{s_2})$. 
\end{claimm}

\begin{proof}
The point $w$ realizes $s_3$ over $\dom s_2 \cup \{v\}$ and therefore witnesses that $s_3$ is Kat\v{e}tov. As for $\Gamma$, it is large relative to $(s_3, \m{D}_{s_2})$ because it is large relative to $(s_2 , \m{D} _{s_2})$ and $(s_3,\m{D}_{s_2}) \leqslant_0 (s_2 , \m{D}_{s_2})$. 
\end{proof}

Observe now that $\min s_3 = s(h) = \min \restrict{s_3}{\m{X} \cup \{ u, v\}} = \min s \leqslant m-1$. Thus, one can apply $\mathcal{H}_{\min s}$ inside $\m{D}_{s_2}$ to $s_3$ and $\m{X} \cup \{ u, v\}$ to obtain $\m{D}_{s_3} \in \binom{\m{D}_{s_2}}{\Ur _p}$ such that $\dom s_3 \cap \m{D}_{s_3} = \m{X} \cup \{ u, v\}$ and $O(\restrict{s_3}{\m{X} \cup \{ u, v\}}, \m{D}_{s_3}) \subset \Gamma$. At that point, both $u$ and $v$ realize $h$ over $\m{X}$ and if $\phi \in \mathcal{F}$, then every point in $\m{D} _{s_3}$ realizing $\phi$ over $\m{X} \cup \{ u\}$ is in $\Gamma$. Thus, we will be done if we can construct $\m{E} \in \binom{\m{D}_{s_3}}{\Ur _p}$ such that:

\begin{itemize}
\item $\left(\m{X} \cup \{u, v \}\right) \cap \m{E} = \m{X} \cup \{ v \}$.
\item For every $\phi \in \mathcal{F}$, every point in $\m{E}$ realizing $\phi$ over $\m{X} \cup \{ v \}$ realizes $\phi$ over $\m{X} \cup \{ u \}$. 
\item Every point in $\m{E}$ realizing $s$ over $\m{X} \cup \{ v \}$ realizes $s_3$ over $\m{X} \cup \{ u, v\}$. 
\end{itemize} 

Once again, this is achieved thanks to Lemma \ref{lem:red}: for $\phi \in \mathcal{F}$, define the map $\hat{\phi}$ on $\m{X}  \cup \{ u , v \}$ by: 
\begin{displaymath}
\left \{ \begin{array}{l}
 \restrict{\hat{\phi}}{\m{X}} = \restrict{\phi}{\m{X}}, \\
 \hat{\phi}(u) = \hat{\phi}(v) = \phi(h).
 \end{array} \right.
\end{displaymath} 

Using the fact that $\phi$ is Kat\v{e}tov over $\m{X}\cup \{h\}$ and $\m{X} \cup \{ u\} \cong \m{X} \cup \{ v\} \cong \m{X} \cup \{ h\}$, it is easy to check that $\hat{\phi}$ is Kat\v{e}tov over $\m{X}  \cup \{ u, v \}$. Let $\widehat{\mathcal{F}} = (\hat{\phi})_{\phi \in \mathcal{F}}$. Working inside $\m{D}_{s_3}$, we would like to apply Lemma \ref{lem:red} to $\m{X} \cup \{ v \} \subset \m{X}  \cup \{ u, v \}$ and the family $\{ s_3\} \cup \widehat{\mathcal{F}}$ to obtain $\m{E}$ as required. It is therefore enough to check: 

\begin{claimm}
For every $g, g' \in \{ s_3\} \cup \widehat{\mathcal{F}}$: \[ \max(\restrict{|g - g'|}{\m{X} \cup \{ v \}}) = \max | g - g'|, \] \[ \min(\restrict{(g+g')}{\m{X} \cup \{ v \}}) = \min(g + g').\] 
\end{claimm}

\begin{proof}
When $g, g' \in \widehat{\mathcal{F}}$, this is easily done. We therefore concentrate on the case where $g = \hat{\phi}$ for $\phi \in \mathcal{F}$ and $g' = s_3$. What we have to do is to show that: \[ |\hat{\phi}(u) - s_3(u)| \leqslant \max(\restrict{|\hat{\phi} - s_3|}{\m{X} \cup \{ v \}}) \ \ (1)\] \[ \hat{\phi}(u) + s_3(u) \geqslant \min(\restrict{(\hat{\phi}+s_3)}{\m{X} \cup \{ v \}}) \ \ (2)\] 

Recall first that $s_3(u) = s(h) + 1$ and that $s_3(v) = s(h)$. Remember also that according to the properties of $\mathcal{F}$, $s(h)\leqslant \phi(h)$. For $(1)$, if $s(h) < \phi(h)$, then we are done since 
\begin{align*} |\hat{\phi}(u) - s_3(u)| & = |\phi(h) - (s(h) + 1)| \\
& = \phi(h) - (s(h) + 1) \\
& \leqslant \phi(h) - s(h) \\
& = \phi(v) - s_3(v) \\
& \leqslant |\hat{\phi}(v) - s_3(v)|.\end{align*} 

On the other hand, if $\phi(h) = s(h)$, then $|\hat{\phi}(u) - s_3(u)| = 1$ but then this is less  than or equal to $\max(\restrict{|\hat{\phi} - s_3|}{\m{X} \cup \{ v \}})$ as this latter quantity is equal to $\max |\phi - s|$, which is at least $1$ since $\phi \in \mathcal{F}$ and $s \notin \mathcal{F}$. Thus, the inequality $(1)$ holds. As for $(2)$, simply observe that \[ \hat{\phi}(u) + s_3 (u) \geqslant \hat{\phi}(v) + s_3 (v). \qedhere\] \end{proof}
This finishes the proof of Lemma \ref{lem:ind2}. \end{proof}

\subsection{Proof of Lemma \ref{lem:red}}

\label{subsection:red}

The purpose of this section is to provide a proof of Lemma \ref{lem:red} which was used extensively in the previous proofs. Let $G_0 \subset G$ be finite subsets of $\Ur_p$, $\mathcal{G}$ a family of Kat\v{e}tov maps with domain $G$ and such that for every $g, g' \in \mathcal{G}$: \[ \max(\restrict{|g - g'|}{G_0}) = \max | g - g'|, \] \[ \min(\restrict{(g+g')}{G_0}) = \min(g + g').\] 

We need to produce an isometric copy $\m{C}$ of $\Ur _p$ inside $\Ur _p$ such that:
\begin{enumerate} 
\item $G \cap \m{C} = G_0$.
\item $\forall g \in \mathcal{G} \ \ O(\restrict{g}{G_0} , \m{C}) \subset O(g, \Ur _p).$  
\end{enumerate} 

First, observe that it suffices to provide the proof assuming that $G$ is of the form $G_0\cup\{ z\}$. The general case is then handled by repeating the procedure. 
  
\begin{lemma}

\label{lem:red1}

Let $\m{X}$ be a finite subspace of $\bigcup \{ O(\restrict{g}{G_0}) : g \in \mathcal{G}\}$. Then there is an isometry $\varphi$ on $\Ur _p$ fixing $G_0 \cup (\m{X} \cap \bigcup \{ O(g) : g \in \mathcal{G}\})$ and such that: \[ \forall g \in \mathcal{G} \ \ \varphi \left( \m{X} \cap O(\restrict{g}{G_0})\right) \subset O(g).\]

\end{lemma}

\begin{proof}
For $x \in \m{X}$, there is a unique element $g_x \in \mathcal{G}$ such that $x \in O(\restrict{g_x}{G_0})$. Let $k$ be the map defined on $G_0 \cup \m{X}$ by
\begin{displaymath}
k(x) = \left \{ \begin{array}{cl}
 d^{\Ur _p}(x,z) & \textrm{if $x \in G_0$,} \\
 g_x(z) & \textrm{if $x \in \m{X}$.}
 \end{array} \right.
\end{displaymath}

\begin{claimm}
The map $k$ is Kat\v{e}tov. 
\end{claimm}

\begin{proof}
The metric space $G_0\cup \{z\}$ witnesses that $k$ is Kat\v{e}tov over $G_0$. Hence, it suffices to check that for every $x \in \m{X}$ and $y \in G_0\cup \m{X}$, \[ |k(x) - k(y)| \leqslant d^{\Ur _p} (x,y) \leqslant k(x) + k(y).\]

Consider first the case $y \in G_0$. Then $d^{\Ur}(x,y) = g_x(y)$ and we need to check that \[ |g_x(z) - d^{\Ur _p}(y,z)| \leqslant g_x(y) \leqslant g_x(z) + d^{\Ur _p}(y,z).\] 

Or equivalently, \[ |g_x(z) - g_x(y)| \leqslant d^{\Ur _p}(y,z) \leqslant g_x(z) + g_x(y).\] 

But this is true since $g_x$ is Kat\v{e}tov over $G_0\cup \{z\}$. Consider now the case $y\in \m{X}$. Then $k(y) = g_y(z)$ and we need to check \[ |g_x(z) - g_y(z)| \leqslant d^{\Ur _p}(x,y) \leqslant g_x(z) + g_y(z).\] 

But since $\m{X}$ is a subspace of $\bigcup \{ O(\restrict{g}{G_0}) : g \in \mathcal{G}\}$, we have, for every $u \in G_0$, \[ |d^{\Ur _p}(x,u) - d^{\Ur _p}(u,y)| \leqslant d^{\Ur _p}(x,y) \leqslant d^{\Ur _p}(x,u) + d^{\Ur _p}(x,u).\]

Since $x \in O(\restrict{g_x}{G_0})$ and $y \in O(\restrict{g_y}{G_0})$, this is equivalent to \[ |g_x(u) - g_y(u)| \leqslant d^{\Ur _p}(x,y) \leqslant g_x(u) + g_y(u).\] 

Therefore, \[ \max (\restrict{|g_x - g_y|}{G_0}) \leqslant d^{\Ur _p}(x,y) \leqslant \min (\restrict{(g_x + g_y)}{G_0}) .\] 

Now, by hypothesis on $\mathcal{G}$, this latter inequality remains valid if $G_0$ is replaced by $G_0\cup \{ z\}$. The required inequality follows. \end{proof}

By ultrahomogeneity of $\Ur _p$ (or, more precisely, by its equivalent reformulation provided in Lemma \ref{prop:extension}), we can consequently realize the map $k$ over $G_0\cup \m{X}$ by a point $z' \in \Ur _p$. The metric space $G_0\cup(\m{X}\cap\bigcup \{ O(g) : g \in \mathcal{G}\})\cup \{k\}$ being isometric to the subspace of $\Ur _p$ supported by $G_0\cup(\m{X}\cap\bigcup \{ O(g) : g \in \mathcal{G}\})\cup \{z\}$, so is the subspace of $\Ur _p$ supported by $G_0\cup(\m{X}\cap\bigcup \{ O(g) : g \in \mathcal{G}\})\cup \{z'\}$. By ultrahomogeneity again, we can therefore find a surjective isometry $\varphi$ of $\Ur _p$ fixing $G_0\cup(\m{X}\cap\bigcup \{ O(g) : g \in \mathcal{G}\})$ and such that $\varphi(z')=z$. Then $\varphi$ is as required: let $g \in \mathcal{G}$ and $x \in O(\restrict{g}{G_0})$. Then: \[d^{\Ur _p}(\varphi(x),z) = d^{\Ur _p}(\varphi(x), \varphi(z')) = d^{\Ur _p}(x,z') = k(x) = g(z). \] 

That is, $\varphi(x) \in O(g)$. \end{proof}

\begin{lemma}

\label{lem:red2}

There is an isometric embedding $\psi$ of $G_0 \cup \bigcup \{ O(\restrict{g}{G_0}) : g \in \mathcal{G})\}$ into $G_0 \cup \bigcup \{ O(g) : g \in \mathcal{G})\}$ fixing $G_0$ such that: \[ \forall g \in \mathcal{G} \ \ \psi \left( O(\restrict{g}{G_0})\right) \subset O(g).\] 

\end{lemma}

\begin{proof}
Let $\{ x_n : n \in \N\}$ enumerate $\bigcup \{ O(\restrict{g}{G_0}) : g \in \mathcal{G})\}$. For $n \in \N$, let $g_n$ be the only $g \in \mathcal{G}$ such that $x_n \in O(\restrict{g_n}{G_0})$. Apply Lemma \ref{lem:red1} inductively to construct a sequence $(\psi _n)_{n \in \N}$ of surjective isometries of $\Ur _p$ such that for every $n\in \N$, $\psi_n$ fixes $G_0 \cup \psi_{n-1}\left(\{x_k : k < n \}\right)$ and $\psi_n(x_n) \in O(g_n)$. Then $\psi$ defined on $G_0 \cup \{ x_n : n \in \N\}$ by $\restrict{\psi}{G_0} = id_{G_0}$ and $\psi(x_n) = \psi_n(x_n)$ is as required. \end{proof}

We now turn to the proof of Lemma \ref{lem:red}. Let $\m{Y}$ and $\m{Z}$ be the metric subspaces of $\Ur _p$ supported by $G \cup \bigcup \{ O(g) : g \in \mathcal{G})\}$ and $G_0 \cup \bigcup \{ O(\restrict{g}{G_0}) : g \in \mathcal{G})\}$ respectively. Let $i_0 : \funct{\m{Z}}{\Ur _p}$ be the isometric embedding provided by the identity. By Lemma~\ref{lem:red2}, the space $\m{Z}$ embeds isometrically into $\m{Y}$ via an isometry $j_0$ that fixes $G_0$. We can therefore consider the metric space $\m{W}$ obtained by gluing $\Ur _p$ and $\m{Y}$ via an identification of $\m{Z} \subset \Ur _p$ and $j_0\left(\m{Z}\right) \subset \m{Y}$. The space $\m{W}$ is described in Figure 1. 

Formally, the space $\m{W}$ can be constructed thanks to a property of  countable metric spaces with distances in $\{ 1, \ldots , p\}$ known as \emph{strong amalgamation}: we can find a countable metric space $\m{W}$ with distances in $\{ 1, \ldots , p\}$ and isometric embeddings $i_1 : \funct{\Ur _p}{\m{W}}$ and $j_1 : \funct{\m{Y}}{\m{W}}$ such that: 
\begin{itemize}
	\item $i_1 \circ i_0 = j_1 \circ j_0$.
	\item $\m{W} = i_1\left(\Ur _p\right) \cup j_1\left(\m{Y}\right)$.
	\item $i_1\left(\Ur _p\right) \cap j_1\left(\m{Y}\right) = (i_1 \circ i_0)\left(\m{Z}\right) = (j_1 \circ j_0)\left( \m{Z}\right)$. 
	\item For every $x \in \Ur _p$ and $y \in \m{Y}$: \begin{align*}d^{\m{W}}(i_1(x),j_1(y)) & = \min \{ d^{\m{W}}(i_1(x),i_1 \circ i_0 (z)) + d^{\m{W}}(j_1 \circ j_0 (z),j_1(y)) : z \in \m{Z}\} \\
& = \min \{ d^{\Ur _p}(x,i_0 (z)) + d^{\m{Y}}(j_0 (z),y) : z \in \m{Z}\} \\
& = \min \{ d^{\Ur _p}(x,z) + d^{\m{Y}}(j_0 (z),y) : z \in \m{Z}\}.
\end{align*}
\end{itemize}
\vskip-5pt
\begin{figure}[h]
\begin{center}
\hskip-10pt\includegraphics[width=132.00mm]{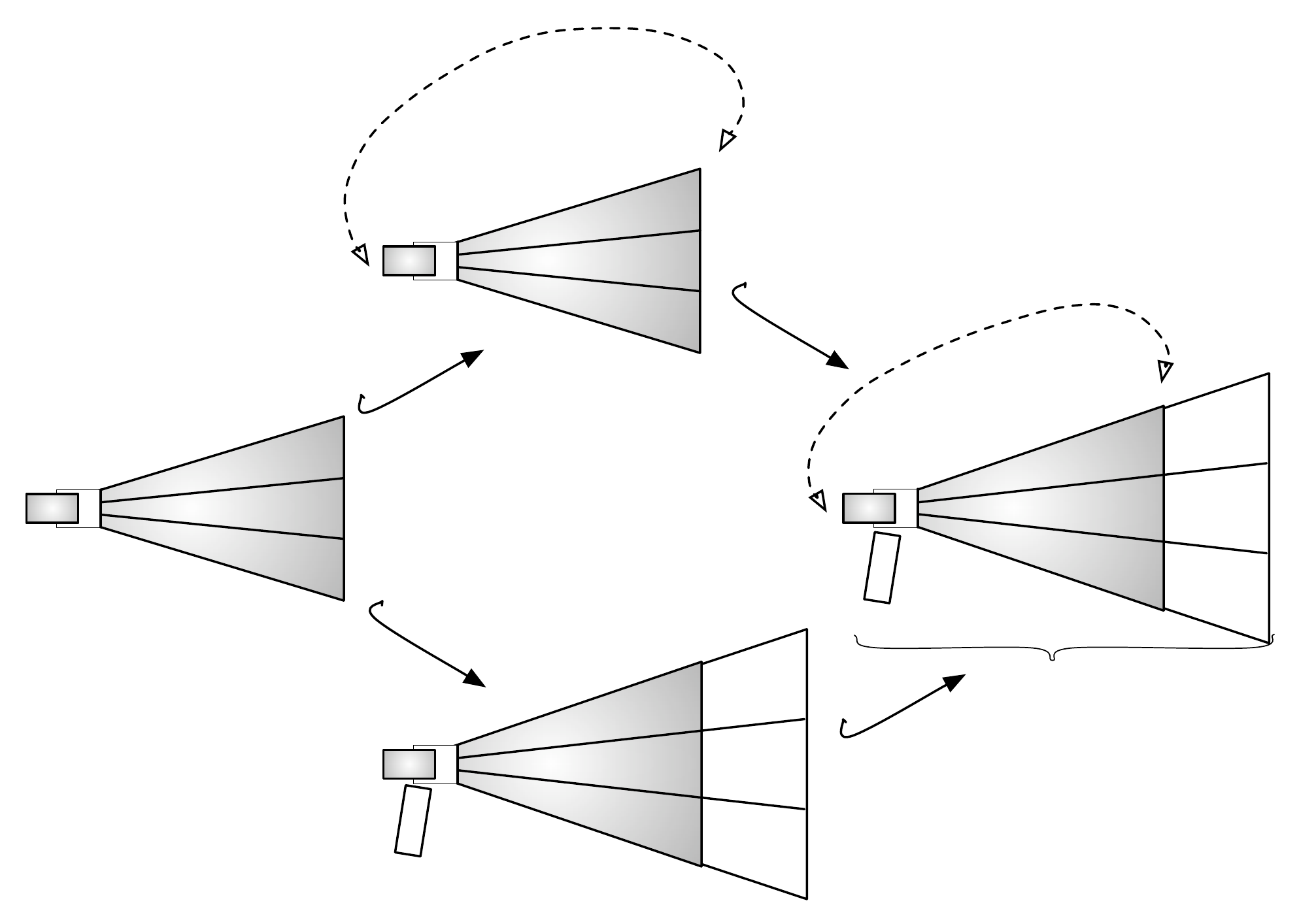}%
\drawat{-75.26mm}{93.04mm}{$\Ur _p$}%
\drawat{-128.91mm}{44.55mm}{$G_0$}%
\drawat{-112.80mm}{45.80mm}{$O(\restrict{g_1}{G_0})$}%
\drawat{-112.80mm}{41.mm}{$O(\restrict{g_2}{G_0})$}%
\drawat{-112.80mm}{37.mm}{$O(\restrict{g_3}{G_0})$}%
\drawat{-93.35mm}{26.26mm}{$j_0$}%
\drawat{-93mm}{19mm}{$G_0$}%
\drawat{-95.19mm}{10mm}{$G$}%
\drawat{-59.50mm}{22.41mm}{$O(g_1)$}%
\drawat{-59.50mm}{14mm}{$O(g_2)$}%
\drawat{-59.50mm}{7mm}{$O(g_3)$}%
\drawat{-54.64mm}{57.79mm}{$i_1$}%
\drawat{-91.30mm}{50.69mm}{$i_0$}%
\drawat{-24.49mm}{64.92mm}{$ i_1 \left(\Ur _p\right)$}%
\drawat{-27.5mm}{23mm}{$\m{W}$}%
\drawat{-40.80mm}{19.mm}{$j_1$}%
\drawat{-50.80mm}{36mm}{$j_1\left(G\right)$}%
\end{center}
\caption{The space $\m{W}$}
\end{figure}

The crucial point here is that in $\m{W}$, every $x \in i_1 \left( \Ur _p\right)$ realizing some $\restrict{g}{G_0}$ over $i_1 \left( G_0\right)$ also realizes $g$ over $j_1 \left( G\right)$.  

Using $\m{W}$, we show how $\m{C}$ can be constructed inductively: Consider an enumeration $\{ x_n : n \in \N\}$ of $i_1\left(\Ur _p\right)$ admitting $i_1\left(G_0\right)$ as an initial segment. Assume that the points $\varphi (x_0),\ldots , \varphi(x_n)$ are constructed so that: 
\begin{itemize}
	\item The map $\varphi$ is an isometry.
	\item $\dom \varphi \subset i_1\left(\Ur _p\right)$.
	\item $\ran \varphi \subset \Ur _p$.
	\item $\varphi(i_1(x)) = x$ whenever $x \in G_0$.
	\item $d^{\Ur _p} (\varphi (x_k) , z) = d^{\m{W}} (x_k , j_1(z))$ whenever $z \in G$ and $k \leqslant n$.
\end{itemize}

We want to construct $\varphi (x _{n+1})$. Consider $e$ defined on $\{ \varphi (x_k) : k \leqslant n\} \cup G $ by: 
\begin{displaymath}
\left \{ \begin{array}{l}
 \forall k \leqslant n \ \ e(\varphi(x_k)) = d^{\m{W}} (x_k , x_{n+1}), \\
 \forall z \in G \ \ e(z) = d^{\m{W}} (j_1(z) , x_{n+1}).
 \end{array} \right.
\end{displaymath}

Observe that the metric subspace of $\m{W}$ given by $\{x_k : k \leqslant n+1\} \cup j_1\left(G\right) $ witnesses that $e$ is Kat\v{e}tov. It follows that the set $E$ of all $y \in \Ur _p$ realizing $e$ over the set $\{ \varphi (x_k) : k \leqslant n\} \cup G $ is not empty and $\varphi (x_{n+1})$ can be chosen in $E$. $\qed$

\section{Appendix}

\label{section:appendix}

The purpose of this section is twofold. First, it is to provide a brief presentation of the attempts to solve the approximate indivisibility problem for $\s$. Then, it is to give a short outline of the proof from \cite{LANVT} according to which the indivisibility of the spaces $\Ur _m$ implies the approximate indivisibility of $\s$. 

The motivation for a combinatorial attack of the approximate indivisibility problem for the Urysohn sphere is based on two ideas. The first one is that the combinatorial point of view
is relevant for the study of countable ultrahomogeneous metric spaces in general. The second idea is that the complete separable ultrahomogeneous metric spaces are
closely linked to the countable ultrahomogeneous metric spaces. This connection is supported by the fact that every complete separable ultrahomogeneous metric space $\m{Y}$ includes a countable ultrahomogeneous dense metric subspace (for a proof, see \cite{LANVT}).

For example, consider the \emph{rational Urysohn space} $\Ur _{\Q}$ which can be defined up to
isometry as the unique countable  ultrahomogeneous metric space with rational distances for which
every   countable metric space with rational distances  embeds isometrically. The Urysohn
space $\Ur$ arises then as the completion of $\Ur _{\Q}$, a fact which is actually essential as it is at
the heart of several important contributions about $\Ur$. In particular, in the original article
\cite{U} of Urysohn, the space $\Ur$ is precisely constructed as the completion of $\Ur _{\Q}$
which is in turn constructed by hand. Similarly, the Urysohn sphere $\s$ arises as the completion of the so-called \emph{rational Urysohn
sphere} $\s _{\Q}$, defined up  to isometry as the unique countable ultrahomogeneous metric space
with distances in $\Q \cap [0,1]$ into which every at most countable metric space with distances in
$\Q \cap [0,1]$ embeds isometrically.

With respect to the approximate indivisibility problem, this latter fact naturally leads to the question of knowing whether $\s _{\Q}$ is indivisible. This question was answered by
to Delhomm\'{e}, Laflamme, Pouzet and Sauer in \cite{DLPS}, where a   detailed analysis of metric
indivisibility is provided and several obstructions to indivisibility are isolated. Cardinality is
such an obstruction: any separable indivisible metric space must be at most countable. Unboundedness is another example: any indivisible metric space must have a bounded distance set. It turns out that $\s
_{\Q}$ avoids those obstacles but encounters a third one: for a metric space $\m{X}$, $x \in
\m{X}$, and $\varepsilon > 0$, let $\lambda _{\varepsilon} (x)$ be the supremum of all reals $l
\leqslant 1$ such that there is an $\varepsilon$-chain $(x_i)_{i \leqslant n}$ containing $x$ and
such that $d^{\m{X}}(x_0 , x_n) \geqslant l$. Then, define
\[\lambda (x) = \inf \{ \lambda_\varepsilon(x) :  \varepsilon > 0 \}.\]
\begin{thm*}[Delhomm\'{e}-Laflamme-Pouzet-Sauer \cite{DLPS}]
\label{thm:lambda divisible} Let $\m{X}$ be a countable metric space. Assume that there is $x_0 \in
\m{X}$ such that $\lambda (x_0) > 0$. Then $\m{X}$ is not indivisible.
\end{thm*}
For $\s_{\Q}$, it is easy to see that ultrahomogeneity  together with the fact that the
distance set contains $0$ as an accumulation point imply that every point $x$ in $\s _{\Q}$ is such
that $\lambda (x) = 1$. It follows that:

\begin{cor*}[Delhomm\'{e}-Laflamme-Pouzet-Sauer \cite{DLPS}]

\label{thm:s_Q divisible}

$\s _{\Q}$ is divisible.

\end{cor*}

This result put an end to the first attempt to solve the oscillation  stability problem for $\s$.
Indeed, had $\s _{\Q}$ been indivisible, $\s$ would have been oscillation stable. But in the
present case, the coloring which was used to divide $\s _{\Q}$ did not lead to any conclusion concerning the approximate indivisibility problem for $\s$.

Later, the idea of using the spaces $\Ur _m$ simply came from the fact that essentially, what makes $\s _{\Q}$ divisible is the richness of its distance set. The hope was then that by working with those simpler spaces, one may be able to avoid the problem encountered above. The results of the present paper show that this hope was justified, but of course, the very first step was to show that the approximate indivisibility property for $\s$ could be captured by the indivisiblity of $\Ur _m$, or, equivalently, by the indivisibility of $\s _m = (\Ur _m, d^{\Ur _m}/m)$. This was one of the purposes of \cite{LANVT}, where the result is achieved by proving the following proposition (section 2.5 in \cite{LANVT}):

\begin{prop}
\label{cor:if s_m indiv then s 1/2m-indiv} Assume that for some strictly positive $m \in \omega$,
$\s _m$ is indivisible. Then $\s$ is $1/m$-indivisible.
\end{prop}

\begin{proof}

This is obtained by showing that there is a separable metric space $\m{Z}$ with distances in $[0,1]$ and
including a copy $\s _m ^*$ of $\s _m$ such that for every $\mc{S}_m \subset \s_m ^*$ isometric to $\s_m$, the set $(\mc{S}_m)_{1/m}$ includes an isometric copy of $\s _{\Q}$. This property indeed suffices to prove the Proposition: the space $\m{Z}$ is separable with distances in $[0,1]$ so by universality of $\s$ we may assume that it is actually a subspace of $\s$. Let now $\gamma$ be a finite partition of $\s$. It induces a finite partition of the copy $\s _m ^*$. By indivisibility of $\s _m$, find $\Gamma \in \gamma$ and $\mc{S} _m \subset \s _m ^*$ such that $\mc{S} _m \subset \Gamma$. By construction of \m{Z}, the set $(\mc{S}_m)_{1/m}$ includes an isometric copy of $\s _{\Q}$. Observe that since the metric completion of $\s _{\Q}$ is $\s$, the closure of $(\mc{S} _m)_{1/m}$ in $\s$ includes a  copy of $\s$, and we are done since $(\mc{S} _m)_{1/m}$ is closed in $\s$. To construct $\m{Z}$, we first construct a metric space $\m{Y}_m$ defined on the set $\s _{\Q}\times \{0,1 \}$ and where the metric $d^{\m{Y}_m}$ satisfies, for every $x, y \in \s _{\Q}$:

\begin{itemize}
	\item $d^{\m{Y}_m}((x,1),(y,1))=d^{\s _{\Q}}(x,y)$.
	\item $d^{\m{Y}_m}((x,0),(y,0)) = \left\lceil d^{\s _{\Q}}(x,y)\right\rceil_m$ (the least $k/m \geqslant d^{\s _{\Q}}(x,y)$).
	\item $d^{\m{Y}_m}((x,0),(x,1))=1/m$.
\end{itemize}

The space $\m{Y}_m$ is really a two-level metric space with a lower level we call $\m{X}_m$. Note that $\m{X}_m$ embeds into $\s _m$ because it is a countable metric space with distances in $\{ k/m : k \in \{ 0, \ldots , m \} \}$. Note also that in $\m{Y}_m$, $(\m{X}_m)_{1/m}$ includes a copy of $\s _{\Q}$. So the basic idea to construct $\m{Z}$ is to start from a copy of $\s _m$, call it $\s_m ^*$, and to use some kind of gluing technique to glue a copy of $\m{Y}_m$ on $\s_m ^*$ along $\mc{X}_m$ whenever $\mc{X}_m$ is a copy of $\m{X}_m$ inside $\s_m ^*$. Because each copy of $\s _m$ contains a copy of $\m{X}_m$, this process adds a copy of $\s _{\Q}$ inside $(\mc{S}_m)_{1/m}$ whenever $\mc{S}_m \subset \s_m ^*$ is isometric to $\s_m$. There is, however, a delicate part. Namely, the gluing process has to be performed in such a way that $\m{Z}$ is separable. For example, this restriction forbids the use of strong amalgamation already used in Lemma \ref{lem:red2}, because then we would go from $\s_m ^*$ to $\m{Z}$ by adding continuum many copies of $\s _{\Q}$ that are pairwise disjoint and at least $1/m$ apart. In spirit, the way this issue is solved is by allowing the different copies of $\s _{\Q}$ we are adding to intersect using some kind of tree-like pattern on the set of copies $\mc{X}_m$ inside $\s_m ^*$. For more details, see \cite{LANVT}. \end{proof}

\end{document}